\documentclass[12pt]{article}
\usepackage{geometry, amsmath, amsthm, nicefrac, mathrsfs, amsfonts, commath, color, graphics}                % See geometry.pdf to learn the layout options. There are lots.
\usepackage{graphicx, tikz}
\usepackage{amssymb}
\usepackage{epstopdf}
\usepackage[normalem]{ulem}
\def\to{\longrightarrow}
\def\mapsto{\longmapsto}

\def\id{\operatorname{Id}}

\def\NN{{\mathbb N}}
\def\QQ{{\mathbb Q}}
\def\RR{{\mathbb R}}
\def\ZZ{{\mathbb Z}}
\def\PP{{\mathbb P}}

\def\op#1{{\operatorname{#1}}}
\def\slzttinv{{\SL{2}\left(\ZZ[t, t^{-1}]\right)}}
\def\SL#1{{\mathbf{SL_{#1}}}}

\definecolor{grey}{gray}{.6}

\newtheorem{theorem}{Theorem}
\newtheorem{lemma}[theorem]{Lemma}

\title{Infinite-dimensional cohomology of $\slzttinv$}
\author{Sarah Cobb}

\newcommand{\Addresses}{{% additional braces for segregating \footnotesize
  \bigskip
  \footnotesize

  S.~Cobb, \textsc{Department of Mathematics, Midwestern State University\\  \rule{5em}{0pt} 3410 Taft Boulevard, Wichita Falls, TX 76308}\par\nopagebreak
  \textit{E-mail address}, \texttt{sarah.cobb@mwsu.edu}
}}

\begin{document}

%\section{}
%\subsection{}

\maketitle

\begin{abstract}
For $J$ an integral domain and $F$ its field of fractions, we construct a map from the 3-skeleton of the classifying space for $\Gamma = \SL{2}(J[t,t^{-1}])$ to a Euclidean building on which $\Gamma$ acts.
%We construct a product of Euclidean buildings on which $\SL{2}(\ZZ[t, t^{-1}])$ acts and a map from the classifying space to this product. 
We then find an infinite family of independent cocycles in the building and lift them to the classifying space, thus proving that the cohomology group $H^2\left(\SL{2}\left(J[t, t^{-1}]\right); F \right)$ is infinite-dimensional.
\end{abstract}

\section{Introduction}

Let $J$ be an integral domain with identity and $F$ its field of fractions. Our goal in this paper is to prove the following theorem:

\begin{theorem} \label{mainthm}
$H^2\left(\SL{2}\left(J[t, t^{-1}]\right); F \right)$ is infinite-dimensional.
\end{theorem}

This theorem generalizes several earlier results about the finiteness properties of $\SL{2}\left(\ZZ[t, t^{-1}]\right)$.
In \cite{KrsticMcCool1997}, Krsti\'c-McCool prove that $\SL{2}(J[t,t^{-1}])$, among other related groups, is not $F_2$. 

In \cite{BuxWortman2006}, Bux-Wortman use geometric methods to prove that $\SL{2}(\ZZ[t,t^{-1}])$ is also not $FP_2$. In particular, they use the action of $\SL{2}(\ZZ[t,t^{-1}])$ on a product of locally infinite trees. Bux-Wortman also ask whether their proof can be extended to show the stronger result that $H_2(\SL{2}(\ZZ[t,t^{-1}]);\ZZ)$ is infinitely generated. Knudson proves that this is the case in \cite{Knudson2008} using algebraic methods. Note that this result also follows from Theorem \ref{mainthm}: Since $\QQ$ is a field, \[H^2(\SL{2}(\ZZ[t,t^{-1}]);\QQ) \simeq \operatorname{Hom}(H_2(\SL{2}(\ZZ[t,t^{-1}]);\QQ),\QQ)\]

Since $\operatorname{Hom}(H_2(\SL{2}(\ZZ[t,t^{-1}]);\QQ),\QQ)$ is infinite dimensional, so is $H_2(\SL{2}(\ZZ[t,t^{-1}]);\QQ)$, which implies that $H_2(\SL{2}(\ZZ[t,t^{-1}]);\ZZ)$ is not finitely generated as a $\ZZ$-module.

The methods in this paper will be geometric. We will define two spaces on which $\SL{2}(J[t,t^{-1}])$ acts: one a Euclidean building as in \cite{BuxWortman2006}, and the other a classifying space for $\SL{2}(J[t,t^{-1}])$. A map between these spaces will allow us to explicitly  define an infinite family of independent cocycles in $H^2(\SL{2}(J[t,t^{-1}]);F)$.

The methods used are based on those of Cesa-Kelly in \cite{CesaKelly2013}, where they are used to show that $H^2\left(\mathbf{SL_3}(\ZZ[t]); \QQ \right)$ is infinite-dimensional. Wortman follows a similar outline in \cite{Wortman2013}.

\subsection*{Acknowledgments}

This paper is adapted from my dissertation, completed at the University of Utah.  I am thankful for support from the institution and its faculty.

I am deeply grateful to Kevin Wortman, my thesis advisor, for guiding me through this work.

I would also like to thank Kai-Uwe Bux for pointing out some errors in an earlier draft; Brendan Kelly and Morgan Cesa, for detailed explanations of their results; and Jon Chaika and Yair Glasner for helpful conversations.

\section{The Euclidean Building}

Throughout, let $J$ be an integral domain with identity, $F$ the field of fractions over $J$, and $\Gamma = \SL{2}(J[t,t^{-1}])$.

%We will begin by recalling two spaces on which $\Gamma = \slzttinv$ acts geometrically and a map between them.

%Construct X and V
We begin by recalling the structure of a Euclidean building on which $\Gamma$ acts. The following construction uses the notation of Bux-Wortman in \cite{BuxWortman2006}. Let $\nu_\infty$ and $\nu_0$ be the valuations on $F(t)$ giving multiplicity of zeros at infinity and at zero, respectively. More precisely, $\nu_\infty\left(\frac{p(t)}{q(t)}\right)= \op{deg}(q(t))-\op{deg}(p(t))$, and $\nu_0\left(\frac{r(t)}{s(t)}t^n\right)=n$, where $t$ does not divide the polynomials $r$ and $s$. Let $T_\infty$ and $T_0$ be the Bruhat-Tits trees associated to $\SL{2}(F(t))$ with the valuations $\nu_\infty$ and $\nu_0$, respectively. We will consider each tree as a metric space with edges having length 1. Let $X=T_\infty \times T_0$.

Since $F((t^{-1}))$ (respectively $F((t))$) is the completion of $F(t)$ with respect to $\nu_\infty$ (resp. $\nu_0$), $\SL{2}(F((t^{-1})))$ (resp. $\SL{2}(F(t))$) acts on the tree $T_\infty$ (resp. $T_0$). Therefore the group $\SL{2}(F((t^{-1}))) \times \SL{2}(F((t)))$ acts on $X$. 

Throughout this paper, we will regard $\Gamma$ and $\SL{2}(F(t))$ as diagonal subgroups of $\SL{2}(F((t^{-1}))) \times \SL{2}(F((t)))$, which act on $X$ via that embedding.

Let $L_\infty$ (resp. $L_0$) be the unique geodesic line in $T_\infty$ (resp. $T_0$) stabilized by the diagonal subgroup of $\SL{2}(F(t))$. Let $\ell_\infty: \RR \to L_\infty$ (resp. $\ell_0: \RR \to L_0$) be an isometry with $\ell_\infty(0)$ (resp. $\ell_0(0)$) the unique vertex with stabilizer $\SL{2}(F[t^{-1}])$ (resp. $\SL{2}(F[t])$). Let $x_0=(\ell_\infty(0), \ell_0(0))$ serve as a basepoint of $X$ and $\Sigma = L_\infty \times L_0$ so that $\Sigma$ is an apartment of $X$. 

\section{The Action of $\Gamma$ on $X$}

The goal of this section is to establish certain large-scale features of the action of $\Gamma$ on $X$. In particular, we will find a horoball containing a sequence of points far from $x_0$ and show that the $\Gamma$-translates of this horoball are disjoint. The techniques here are similar to those used by Bux-Wortman in \cite{BuxWortman2011}.

For certain parts of the proof, it will be convenient to work with $\Gamma_F=\SL{2}(F[t,t^{-1}])$ instead of with $\Gamma$. Note that $\Gamma_F$ also acts on $T_\infty$ and on $T_0$ and therefore acts diagonally on $X$.

We also define the following useful subgroups:
\begin{equation*}
P = \left\{\left.\begin{pmatrix} a & b \\ 0 & a^{-1}\end{pmatrix}\right\vert a,b \in F(t)\right\}
\end{equation*}

\begin{equation*}
U = \left\{\left.\begin{pmatrix}1 & x \\ 0 & 1\end{pmatrix}\right\vert x \in F(t)\right\}
\end{equation*}

\begin{equation*}
A = \left\{\left.\begin{pmatrix}a & 0 \\ 0 & a^{-1}\end{pmatrix}\right\vert a \in F(t)\right\}
\end{equation*}

For any diagonal subgroup $G$ of $\SL{2}(F((t^{-1}))) \times \SL{2}(F((t)))$, we will use the notation
\begin{equation*}
G_\Gamma = G \cap \Gamma
\end{equation*}
and 
\begin{equation*}
G_{\Gamma_F} = G \cap \Gamma_F
\end{equation*}

\begin{lemma}\label{doublecosets}
The double coset space $\Gamma_F \backslash \SL{2}(F(t))/P$ is a single point.
\end{lemma}

\begin{proof}
Consider the action of $\SL{2}(F(t))$ on $\PP^1(F(t))$ by $\begin{pmatrix} a & b \\ c & d \end{pmatrix} \begin{bmatrix} x \\ y \end{bmatrix} = \begin{bmatrix} ax + by \\ cx + dy \end{bmatrix}$. Under this action, $P$ is the stabilizer of $\begin{bmatrix} 1 \\ 0 \end{bmatrix}$, so $\SL{2}(F(t))/P = \PP^1(F(t))$. Since $\Gamma_F$ acts transitively on $\PP^1(F(t))$, $\Gamma_F \backslash \SL{2}(F(t))/P$ is a single point.
\end{proof}

Let $\beta_\infty$ be the Busemann function in $T_\infty$ for $\ell_\infty\vert_{[0,\infty)}$ and $\beta_0$ the Busemann function in $T_0$ for $\ell_0\vert_{[0,\infty)}$

Let $\rho:{[0,\infty)} \to X$ be the geodesic ray with $\rho(s) = (\ell_\infty(s), \ell_0(s))$. Let $\beta_\rho$ be the Busemann function for $\rho$ with $\beta_\rho(x_0) = 0$ such that $\beta_\rho(x,y) = \beta_\infty(x) + \beta_0(y)$ for any $x \in T_\infty$ and $y \in T_0$. Thus $\beta_\rho^{-1}([d,\infty))$ is a horoball based at $\rho(\infty)$.

\begin{lemma}\label{psstabilizer}
$P_{\Gamma_F}$ stabilizes the horosphere $\beta_\rho^{-1}(r)$ for every $r \in \RR$.
\end{lemma}

\begin{proof}

Since $P_{\Gamma_F} = A_{\Gamma_F} U_{\Gamma_F}$, it suffices to show that each of these groups preserves $\beta_\rho^{-1}(r)$.

Since $\beta_\infty$ and $\beta_0$ are $U_{\Gamma_F}$-equivariant, so is $\beta_\rho$. Thus $U_{\Gamma_F}$ preserves $\beta_\rho^{-1}(r)$.

Let 

\begin{equation*}
A_F = \left\{\left.\begin{pmatrix}a & 0 \\ 0 & a^{-1}\end{pmatrix}\right\vert a \in F \right\}
\end{equation*}

and $D = \begin{pmatrix} t & 0 \\0 & t^{-1} \end{pmatrix}$. Then $A_{\Gamma_F} = A_F D^\ZZ$

For any $x \in T_\infty, y \in T_0$ we have $\beta_\infty(Dx) = \beta_\infty(x) + 2$ and $\beta_0(Dy) = \beta_0(y) - 2$. Therefore $\beta_\rho(D(x,y)) = \beta_\rho(x,y)$ and $D$ preserves $\beta_\rho^{-1}(r)$.

Since $A_F$ acts trivially on $X$, it preserves $\beta_\rho^{-1}(r)$.

\end{proof}

Note that this also implies that the horoball $\beta_\rho^{-1}([d,\infty))$ is stabilized by $P_{\Gamma_F}$, since it is a union of horospheres.

The following lemma will be useful for describing the action of $U$ on the boundaries $\partial T_\infty$ and $\partial T_0$.

\begin{lemma}\label{opensetsintersect}
For any open sets $V \subseteq F((t^{-1}))$ and $W \subseteq F((t))$, we have 
\begin{equation*}
V \cap W \cap F(t) \neq \emptyset
\end{equation*}
\end{lemma}

\begin{proof}
Consider the sequence $\alpha_i = \frac{t^i}{t^i + 1}$, which converges to 1 in $F((t^{-1}))$ and to 0 in $F((t))$. Let $v \in V$ and $w \in W$. Then the sequence $\beta_i = (v-w)\alpha_i + w$ converges to $v$ in $F((t^{-1}))$ and to $w$ in $F((t))$. Therefore $\beta_i \in V \cap W \cap F(t)$ for any sufficiently large $i$.
\end{proof}

\begin{lemma}\label{sigmafundamental}
$\Sigma$ is a fundamental domain for the action of $U$ on $X$. 
\end{lemma}

\begin{proof}

Let $x \in T_\infty$, $y \in T_0$. Then the set 
Ä
\begin{equation*}
V = \left\{ \alpha \in F((t^{-1})) \left\vert \begin{pmatrix} 1 & \alpha \\ 0 & 1 \end{pmatrix} \ell_\infty \text{ passes through } x\right.\right\}
\end{equation*}

is open in $F((t^{-1}))$. Similarly, 

\begin{equation*}
W = \left\{ \beta \in F((t)) \left\vert \begin{pmatrix} 1 & \beta \\ 0 & 1 \end{pmatrix} \ell_0 \text{ passes through } y\right.\right\}
\end{equation*}
is open in $F((t))$.

By lemma \ref{opensetsintersect}, there is some $\gamma \in V \cap W \cap F(t)$. Since $(x,y) \in \begin{pmatrix} 1 & \gamma \\ 0 & 1 \end{pmatrix} \Sigma$, we have $U\Sigma = X$.

To see that $\Sigma$ contains only one point from each $U$-orbit, observe that a point $x \in \Sigma$ is uniquely determined by the values $\beta_\infty(x)$ and $\beta_0(x)$. Since $U$ preserves $\beta_\infty$ and $\beta_0$, $Ux \cap \Sigma = x$.

\end{proof}

Let $x_n = (\ell_\infty(n), \ell_0(n))$.

\begin{lemma}\label{orbithoroball}
There is some $s>0$ such that for every $R > 0$ 
\begin{equation*}
\beta_\rho^{-1}(R) \subset \op{Nbhd}_{s}(\Gamma_F x_m)
\end{equation*}
where $m$ is an integer such that $\vert R-m \vert \leq \frac{1}{2}$.
\end{lemma}

\begin{proof}

Let $y = (\ell_\infty(R), \ell_0(R)) = \beta_\rho^{-1}(R) \cap \rho$ and $D = \begin{pmatrix} t & 0 \\ 0 & t^{-1} \end{pmatrix}$.

Since 

\begin{equation*}
\beta_\rho^{-1}(R) \cap \Sigma = \left\{(\ell_\infty(x), \ell_0(R-x)) \left\vert x \in \RR\right.\right\}
\end{equation*}
and
\begin{equation*}
D^n \cdot (\ell_\infty(a), \ell_0(b)) =  (\ell_\infty(a+2n), \ell_0(b-2n))
\end{equation*}
 we see that 
\begin{equation*}
\beta_\rho^{-1}(R) \cap \Sigma \subset
\op{Nbhd}_{2\sqrt{2}}\left(A_\Gamma \cdot y \right)
\end{equation*}

Since $X = U \Sigma$ and $U$ preserves $\beta_\rho^{-1}(R)$,
\begin{equation*}
\beta_\rho^{-1}(R) \subset \op{Nbhd}_{2\sqrt{2}} \left(U A_\Gamma \cdot y \right)
\end{equation*}

Let $m \in \ZZ$ be such that $|R-m| \leq \frac{1}{2}$.
Since $d(x_m, y) < \frac{\sqrt{2}}{2}$, we have
\begin{equation*}
\beta_\rho^{-1}(R) \subset \op{Nbhd}_{2\sqrt{2} + \frac{\sqrt{2}}{2}}(U A_\Gamma \cdot x_m)
\end{equation*}

Let 
\begin{equation*}
C = \left\{\left. \left( \begin{pmatrix} 1 & a \\ 0 & 1\end{pmatrix}, \begin{pmatrix} 1 & b \\ 0 & 1\end{pmatrix}\right) \right\vert a \in F[[t^{-1}]], b \in F[[t]]\right\}
\end{equation*}

Then $C \subset \SL{2}(F((t^{-1}))) \times \SL{2}(F((t)))$ is bounded with $U \subset C U_\Gamma$. 
Thus, there is some $r \in \RR$ such that $Ux \subset \op{Nbhd}_r (U_\Gamma x)$ for every $x \in X$, which yields
\begin{equation*}
\beta_\rho^{-1}(R) \subset \op{Nbhd}_{2\sqrt{2} + \frac{\sqrt{2}}{2} + r}(U_\Gamma A_\Gamma \cdot x_m)
\end{equation*}

Since $U_\Gamma A_\Gamma < \Gamma_F$, this gives the desired inclusion

\begin{equation*}
\beta_\rho^{-1}(R) \subset \op{Nbhd}_{2\sqrt{2} + \frac{\sqrt{2}}{2} + r}(\Gamma_F \cdot x_m)
\end{equation*}

\end{proof}

\begin{lemma}\label{unboundedsequence}
For every $C>0$, there is some $N \in \NN$ such that $d(x_n, \Gamma_F x_0)>C$ for every $n > N$.
\end{lemma}

\begin{proof}
The following proof is based on the proof of Lemma 2.2 by Bux-Wortman in \cite{BuxWortman2006}.

Let $C > 0$. 

We will show that any subsequence of $\{x_{2n}\}$ is unbounded in the quotient space $\Gamma_F \backslash X$. This implies that only a finite number of points of $\{x_{2n}\}$ can be contained in any neighborhood of $\Gamma_F x_0$. Since a finite number of points of $\{x_{2n}\}$ are contained in $\operatorname{Nbhd}_{C+\sqrt{2}}(\Gamma_F x_0)$ and $d(x_{2n}, x_{2n+1}) = \sqrt{2}$, a finite number of points of $\{x_{2n+1}\}$ are contained in $\operatorname{Nbhd}_C(\Gamma_F x_0)$. This suffices to prove the lemma.

The group $\SL{2}(F(t)) \times \SL{2}(F(t))$ acts componentwise on $X$, and has a metric induced by the valuations $v_\infty$ and $v_0$. Under this metric, the stabilizer of $x_0$ is a bounded subgroup. Thus, to prove that a set of vertices in $\Gamma_F \backslash X$ is not bounded, it suffices to prove that it has unbounded preimage under the projection
\begin{equation*}
\Gamma_F \backslash \left(\SL{2}(F(t)) \times \SL{2}(F(t))\right) \to \Gamma_F \backslash X
\end{equation*}
given by $\Gamma_F g \mapsto \Gamma_F g x_0$.

Let $D = \begin{pmatrix} t & 0 \\ 0 & t^{-1} \end{pmatrix}$ and \mbox{$\delta = (D, D^{-1}) \in \SL{2}(F(t)) \times \SL{2}(F(t))$.} Then \mbox{$A^n x_0 = x_{2n}$.} It therefore suffices to prove that any infinite subset of $\{\Gamma_F \delta^n\}_{n \in \NN}$ is unbounded in $\Gamma_F \backslash \left(\SL{2}(F(t)) \times \SL{2}(F(t))\right)$.

Assume that this is not the case. That is, assume that there is some infinite subset $I \subset \NN$ such that $\{\Gamma_F \delta^i\}_{i \in I}$ is bounded. Then there is some $L$ such that for every $i \in I$ there is some $M_i = \begin{pmatrix} a_i & b_i \\ c_i & d_i \end{pmatrix} \in \Gamma_F$ such that the values of $v_\infty$ of the coefficients of $M_iD^i$ are bounded from below by $L$ and the values of $v_0$ of the coefficients of $M_iD^{-i}$ are also bounded from below by $L$. Then
\begin{equation*}
L \leq v_\infty(a_nt^n) = v_\infty(a_n) + n v_\infty (t) = v_\infty(a_n)-n
\end{equation*}
and
\begin{equation*}
L \leq v_0(a_nt^{-n}) = v_0(a_n) - n v_0 (t) = v_0(a_n)-n
\end{equation*}
This gives that $v_\infty(a_n) \geq 1$ and $v_0(a_n) \geq 1$ whenever $n \geq 1-L$, which implies that $a_n = 0$.

The same argument shows that $c_n = 0$, which implies that $M_i$ is not in $\Gamma_F$ when $i \geq 1-L$.
\end{proof}

\begin{lemma}\label{farawayhoroballs}
There is some $R>0$ such that $\Gamma_F \cdot x_0 \cap \beta_\rho^{-1}([R,\infty)) = \emptyset$.
\end{lemma}

\begin{proof}
Since $\beta_\rho^{-1}([R,\infty))$ is a union of the horospheres $\beta_\rho^{-1}(r)$ for $r \in [R,\infty)$, Lemma \ref{orbithoroball} guarantees that
 \begin{equation*}
 \beta_\rho^{-1}([R,\infty)) \subset \bigcup_{\left\{m \in \ZZ | m > R-\frac{1}{2}\right\}} \op{Nbhd}_{s}(\Gamma_F \cdot x_m)
 \end{equation*}
 for some $s$.
 
 By Lemma \ref{unboundedsequence}, $d(x_m, \Gamma_F \cdot x_0) > s$ for large values of $R$. This implies that 
 \begin{equation*}
 d(\Gamma_F \cdot x_m, \Gamma_F \cdot x_0) > s
 \end{equation*}
 
 Thus, we may conclude that 

\begin{equation*}
\Gamma_F \cdot x_0 \cap \beta_\rho^{-1}([R,\infty)) = \emptyset
\end{equation*}
\end{proof}

Fix $R$ as in Lemma \ref{farawayhoroballs}. Let $d$ be the maximum distance from a point in the horosphere $\beta_\rho^{-1}(R)$ to the orbit $\Gamma_F \cdot x_0$.

\begin{lemma}\label{disfinite}
$d$ is finite.
\end{lemma}

\begin{proof}

By Lemma \ref{orbithoroball}, $\beta_\rho^{-1}(R) \subset \op{Nbhd}_s(\Gamma_F x_n)$ for some $n > 0$, $s>0$. Thus for every $x \in \beta_\rho^{-1}(R)$ there is some $\gamma \in \Gamma_F$ such that $d(x,\gamma \cdot x_n) < s$. Since 
\begin{equation*}
d(\gamma \cdot x_n, \gamma \cdot x_0) = d(x_n, x_0) = n\sqrt{2}
\end{equation*}
 we have $d(x,\gamma \cdot x_0) < s + n\sqrt{2}$. Since this bound is independent of $x$, it follows that $d < s + n\sqrt{2}$.
\end{proof}

Let $H=\beta_\rho^{-1}([R+d, \infty))$.

\begin{lemma}\label{nooverlap}
Let $\gamma \in \Gamma_F$. If $\gamma H \cap H \neq \emptyset$, then $\gamma H = H$ and $\gamma \in P_{\Gamma_F}$.
\end{lemma}

\begin{proof}
Let $C$ be the chamber of  $\partial X$ containing the point $\rho(\infty)$ and $S \subset X$ an apartment whose boundary contains $C$ and $\gamma C$.
The endpoints of the arc $C$ are $\ell_\infty(\infty)$ and $\ell_0(\infty)$. Thus, any chamber adjacent to $C$ contains one of these two points. Because $\op{Stab}_{\Gamma_F}(\ell_\infty(\infty)) = \op{Stab}_{\Gamma_F}(\ell_0(\infty)) < P$, any element of $\Gamma_F$ either stabilizes both $\ell_\infty(\infty)$ and $\ell_0(\infty)$ or neither. Therefore, $C$ and $\gamma C$ cannot be adjacent. Since each apartment of $\partial X$ contains exactly four chambers, the two chambers are either equal or opposite.

If $C$ and $\gamma C$ are opposite, then $S \cap H \cap \gamma H$ is contained in a neighborhood of a hyperplane in $S$. Suppose $x \in S \cap \gamma H \cap \beta_\rho^{-1}(R)$. It follows from the choice of $d$ that there is some $y \in \Gamma_F \cdot x_0$ such that $d(x,y) \leq d$. Since $\beta_{\gamma\rho}(x) \geq R+d$, it is clear that $\beta_{\gamma\rho}(y) \geq R$. Therefore $\beta_\rho(\gamma^{-1}y) \geq R$, which contradicts Lemma \ref{farawayhoroballs}, since $\gamma^{-1}y \in \Gamma_F \cdot x_0$.

Thus $\gamma C = C$, which implies that $\gamma \in P_{\Gamma_F}$. It follows that $H$ and $\gamma H$ are horoballs based at the same boundary point. Since $\gamma$ preserves distance from $\Gamma_F \cdot x_0$, $\gamma H = H$.
\end{proof}

\section{The Classifying Space}

%Construct $X_3$
We will now construct a second space on which $\Gamma$ acts freely and properly discontinuously by isometries. A map between the two spaces will allow us to 
construct cohomology classes of $\Gamma$ in the familiar context of a product of trees.

Let $X_0$ be a discrete collection of points $\{x_\gamma \vert \gamma\in\Gamma\}$. $\Gamma$ acts freely on this set by $\gamma^\prime\cdot x_\gamma = x_{\gamma^\prime\gamma}$. Define a $\Gamma$-equivariant map $\psi_0: X_0 \to X$ by $\psi_0(x_\gamma) = \gamma\cdot x_0$. By Lemma \ref{farawayhoroballs}, $\psi_0(X_0) \cap H = \emptyset$. Note also that $\psi_0(X_0) \cap \gamma H = \emptyset$ for every $\gamma \in \Gamma$.

Construct $X_1$ from $X_0$ by attaching a 1-cell $E_{x_\gamma, x_{\gamma^\prime}}$ of length 1 between every pair of points $x_\gamma, x_{\gamma^\prime} \in X_0$. For $\zeta \in \Gamma$, let $\zeta: E_{x_\gamma,x_{\gamma^\prime}} \to E_{x_{\zeta\gamma},x_{\zeta \gamma^\prime}}$ be the unique distance-preserving map with $\zeta(x_\gamma) = x_{\zeta \gamma}$ and $\zeta(x_{\gamma^\prime}) = x_{\zeta\gamma^\prime}$. This defines a $\Gamma$-action on $X_1$ that extends the $\Gamma$-action on $X_0$.

We wish to define a $\Gamma$-equivariant map $\psi_1: X_1 \to X$ extending $\psi_0$ and with
\begin{equation*}
\psi_1(X_1) \subset X - \bigcup_{\gamma \in \Gamma} \gamma H
\end{equation*}

For every nonidentity element $\gamma \in \Gamma$, choose a path $c_\gamma$ from $x_0$ to $\gamma x_0$. Since $\partial (\gamma H)$ is connected for all $\gamma \in \Gamma$, we can choose all such paths to lie outside of $\bigcup\limits_{\gamma \in \Gamma} \gamma H$. We may then define $\psi_1(E_{x_\gamma, x_{\gamma^\prime}})$ to lie along the path $\gamma c_{(\gamma^{-1}\gamma^\prime)}$.

We will build a family of spaces $X_n$ inductively. Beginning with an $n$-dimensional, \mbox{$(n-1)$-connected} cell complex $X_{n}$ on which $\Gamma$ acts freely and a $\Gamma$-equivariant map $\psi_{n}: X_{n} \to X$, we may construct an $n+1$-dimensional, $n$-connected cell complex $X_{n+1}$ on which $\Gamma$ acts freely and a $\Gamma$-equivariant map $\psi_{n+1}:X_{n+1} \to X$. For every map $s: S^{n} \to X_{n}$, attach an $(n+1)$-cell $E_s$ along $\gamma s$ for every $\gamma \in \Gamma$. This results in an $n$-connected space $X_{n+1}$. Define a $\Gamma$-action on $X_{n+1}$ by $\gamma E_s = E_{\gamma s}$. For each $s$ as above, define $\psi_{n+1,s}: E_s \to X$ to extend $\psi_{n}(\partial E_s)$ with $\psi_{n+1,\gamma s} = \gamma\psi_{n+1,s}$. Because $\pi_{n+1}(X)$ is trivial for all $n$, this is always possible. There is a unique $\Gamma$-equivariant map $\psi_{n+1}$ extending these maps.

Since the focus of Theorem \ref{mainthm} is second-dimensional cohomology, we will use the space $X_3$ extensively. Therefore we will let $\psi=\psi_3$. Since $\Gamma \backslash X_3$ is the 3-skeleton of a $K(\Gamma,1)$, $H^2(\Gamma; F) = H^2(X_3; F)$.

\section{Local Cohomology}
Before defining cocyles on $\Gamma \setminus X_3$, we will define cocycles in the relative homology groups of several subspaces of $X$.

Recall that $x_n = (\ell_\infty(n), \ell_0(n))$. We will assume from now on that $n>R+d$ so that $x_n \in H$. Recall also
\begin{equation*}
U = \left\{\left.\begin{pmatrix}1 & x \\ 0 & 1\end{pmatrix}\right\vert x \in F(t)\right\}
\end{equation*}
and $U_\Gamma = U \cap \Gamma$. 
 
Let 
\begin{equation*}
U_n=\left\{\left. \begin{pmatrix}1 & x \\ 0 & 1\end{pmatrix}\right\vert v_\infty(x) \geq -n, v_0(x) \geq -n \right\}
\end{equation*}

and

\begin{equation*}
U^n=\left\{\left. \begin{pmatrix}1 & \sum_{i=k}^\ell a_i t^i \\ 0 & 1\end{pmatrix}\right\vert k,\ell \in \ZZ, a_i = 0 \text{ for } -n\leq i \leq n \right\}
\end{equation*}

Then $U_n \cap U^n = \operatorname{Id}$ and $U_\Gamma = U_n U^n$.
 
Let $S_n$ be the star of $x_n$ in $X$ (that is, the collection of cells having $x_n$ as a vertex) and $C$ the cell in $S_n$ containing $x_{n-1}$. Let $S_n^\downarrow = U_n C$. This is also the union of 2-cells in $S_n$ with $\beta_\infty$ and $\beta_0$ bounded above by $n$.

We will define cocycles $\varphi_n \in H^2(S_n^\downarrow, S_n^\downarrow\cap\partial S_n; F)$. Summing these over cosets of $U_\Gamma$ in $\Gamma$ will give us cocycles in $H^2\left(\Gamma; F \right)$, which we will use to prove Theorem \ref{mainthm}.

\begin{lemma}\label{actstransitively}
$\left\{\left.\begin{pmatrix} 1 & at^{-n} + bt^n \\ 0 & 1\end{pmatrix}\right\vert a, b \in F \right\} < U_n$ acts transitively on the set of 2-cells in $S_n^\downarrow$.
\end{lemma}
\begin{proof}

We begin by describing subsets of $U_n$ that act transitively on the sets 
\begin{equation*}
E_0 = \left\{\ell_\infty(n) \times e \vert e \subset T_0 \text{ an edge}, \ell_0(n) \in e, \ell_0(n+1) \not\in e\right\}
\end{equation*}
 and
\begin{equation*}
E_\infty = \left\{e \times \ell_0(n) \vert e \subset T_\infty \text{ an edge}, \ell_\infty(n) \in e, \ell_\infty(n+1) \not\in e\right\}
\end{equation*}

Let
\begin{equation*}
U_n^{E_0} = \left\{\left.\begin{pmatrix}1 & at^{-n} \\ 0 & 1\end{pmatrix}\right\vert a \in F\right\}
\end{equation*}
and 
\begin{equation*}
U_n^{E_\infty} = \left\{\left.\begin{pmatrix}1 & bt^{n} \\ 0 & 1\end{pmatrix}\right\vert b \in F\right\}
\end{equation*}

$U_n^{E_0}$ acts transitively on the set of edges in $T_0$ incident to $\ell_0(n)$ and not to $\ell_0(n+1)$. Since $U_n^{E_0}$ stabilizes $x_n$, it acts transitively on $E_0$. Denote by $e_0$ the edge in $E_0$ between $x_n$ and $(\ell_\infty(n), \ell_0(n-1))$. Let $e_a = \begin{pmatrix}1 & at^{-n} \\ 0 & 1\end{pmatrix} e_0$ for each $a \in F$.

Similarly, $U_n^{E_\infty}$ acts transitively on the set of edges in $T_\infty$ incident to $\ell_\infty(n)$ and not to $\ell_\infty(n+1)$. Since $U_n^{E_\infty}$ stabilizes $x_n$, it acts transitively on $E_\infty$. Denote by $f_0$ the edge in $E_\infty$ between $x_n$ and $(\ell_\infty(n-1), \ell_0(n))$. Let $f_b = \begin{pmatrix}1 & bt^{n} \\ 0 & 1\end{pmatrix} f_0$ for each $b \in F$.

Since any 2-cell in $S_n^\downarrow$ contains a unique pair of edges $e_a$ and $f_b$ in its boundary, 
\begin{equation*}
 U_n^{E_0} \times U_n^{E_\infty} = \left\{\left.\begin{pmatrix}1 & at^{-n} + bt^{n} \\ 0 & 1\end{pmatrix}\right\vert a,b \in F \right\}
 \end{equation*}
acts transitively on the set of 2-cells in $S_n^\downarrow$ as well as on $E_0$ and $E_\infty$.
\end{proof}

 Let $C^n_{0,0}$ be the 2-cell in $S_n^\downarrow$ containing $x_{n-1}$, and let $C^n_{a,b} = \begin{pmatrix}1 & at^{-n} + bt^{n} \\ 0 & 1\end{pmatrix} C^n_{0,0}$. Orient each edge $e_a$ and $f_b$ with initial vertex $x_n$. We may then choose an orientation for each $C^n_{a,b}$ such that we have $\partial C^n_{a,b} = e_a - f_b + D$ for some chain $D \subset S_n^\downarrow \cap \partial S_n$.

Define a cochain $\varphi_n$ by $\varphi_n(C^n_{a,b})=ab$. Since $X$ is 2-dimensional, any 2-cochain is a cocycle.

We will define cocycles in $H^2(\Gamma \backslash X_3; F)$ as sums of the $\varphi_n$'s. The following lemma will be used in proving that these cocycles are well-defined.

\begin{lemma}\label{invariant} $\varphi_n$ is $U_n$-invarant. \end{lemma}
\begin{proof}
Let a \emph{basic cycle} in $S_n^\downarrow$ be one of the form $C^n_{x,y} - C^n_{x^\prime,y} - C^n_{x,y^\prime} + C^n_{x^\prime,y^\prime}$. We will show first that $\varphi_n$ is $U_n$-invariant on basic cycles, then that any cycle is a sum of basic cycles.

Let 
\begin{equation*}
u = \begin{pmatrix}1 & \sum_{i=-n}^{n}a_i t^i \\ 0 & 1\end{pmatrix} \in U_n
\end{equation*}

and $D$ be a basic cycle.
\begin{align*}
\varphi_n(uD) &= \varphi_n(u(C^n_{x,y} - C^n_{x^\prime,y} - C^n_{x,y^\prime} + C^n_{x^\prime,y^\prime})) \\
&= \varphi_n(uC^n_{x,y} - uC^n_{x^\prime,y} - uC^n_{x,y^\prime} + uC^n_{x^\prime,y^\prime}) \\
&= \varphi_n(C^n_{x+a_{-n},y+a_{n}} - C^n_{x^\prime+a_{-n},y+a_{n}} - C^n_{x+a_{-n},y^\prime+a_{n}} + C^n_{x^\prime+a_{-n},y^\prime+a_{n}}) \\
&= (x+a_{-n})(y+a_{n}) - (x^\prime+a_{-n})(y+a_{n}) - (x+a_{-n})(y^\prime+a_{n})\\
& \hspace{2in} + (x^\prime+a_{-n})(y^\prime+a_{n}) \\
&= xy - x^\prime y - x y^\prime  + x^\prime y^\prime  \\
&= \varphi_n(D)
\end{align*}
Thus $\varphi_n$ is $U_n$-invariant on basic cycles.

It remains to show that all cycles are sums of basic cycles. Define the length $l$ of a chain in $H_2(S_n^\downarrow, S_n^\downarrow \cap \partial S_n)$ by
\begin{equation*}
l\left(\sum_{i,j \in J} \alpha_{i,j}C^n_{i,j}\right) = \sum_{i,j \in J} |\alpha_{i,j}|
\end{equation*}
Suppose there are cycles which are not sums of basic cycles. Let $B = \sum_{i,j \in J} \alpha_{i,j}C^n_{i,j}$ be such a cycle with the property that $l(B) \leq l(B^\prime)$, for $B^\prime$ any other such cycle. Let $C^n_{x,y}$ be a 2-cell such that $\alpha_{x,y} > 0$. Since $B$ is a cycle, $\partial B=0$. Therefore, there are some $x^\prime, y^\prime \in J$ such that $\alpha_{x^\prime,y}, \alpha_{x,y^\prime} < 0$. Then
\begin{equation*}
l(B-(C^n_{x,y} - C^n_{x^\prime,y} - C^n_{x,y^\prime} + C^n_{x^\prime,y^\prime})) \leq l(B)-2
\end{equation*}
Since this cycle differs from $B$ by a basic cycle, it cannot be written as a sum of basic cycles, contradicting the assumption that $B$ is the shortest cycle with that property.
\end{proof}

\section{Cohomology}

We now wish to define cocycles in $H^2(\Gamma \backslash X_3; F)$ by summing $\varphi_n$ over cosets of $U_\Gamma$ in $\Gamma$. 

Let $Y_n = \operatorname{Stab}_U(x_n) \cdot \Sigma$

\begin{lemma}
$Y_n$ is a fundamental domain for the action of $U^n$ on $X$.
\end{lemma}

\begin{proof}
We will show first that $Y_n$ intersects every $U^n$-orbit in $X$. That is,
\begin{equation*}
X = U^n \operatorname{Stab}_U(x_n) \Sigma
\end{equation*}
Since $U \Sigma = X$, it suffices to show that $U x_n = U^n x_n$. Let $u \in U$. Since $u$ preserves $\beta_\rho$, there is some $m>n$ such that $ux_m = x_m$. By lemma \ref{actstransitively}, for every $n<i\leq m$, the group 
\begin{equation*}
V_i = \left\{\left.\begin{pmatrix}1 & a_{-i}t^{-i} + a_i t^i \\ 0 & 1\end{pmatrix}\right\vert a_{-i},a_i \in F \right\}
\end{equation*}
acts transitively on the set of 2-cells in $S_i^\downarrow$. In particular, it acts transitively on $\beta_\rho^{-1}(i-1) \cap B_{\sqrt{2}}(x_i)$ where $B_r(y)$ denotes the sphere of radius $r$ centered at $y$. It follows that $V_m \cdot V_{m-1} \dots V_{n+1}$ acts transitively on $\beta_\rho^{-1}(n) \cap B_{\sqrt{2}(m-n)}(x_m)$.

Since $U$ acts by $\beta_\rho$-preserving isometries, $ux_n \in \beta_\rho^{-1}(n) \cap B_{\sqrt{2}(m-n)}(x_m)$. Thus, there is some $u^\prime \in V_m \cdot V_{m-1} \dots V_{n+1}$ with $u^\prime x_n = u x_n$.

It remains to show that $Y_n$ contains exactly one point of each $U^n$-orbit.

Suppose there is some $u^n = \begin{pmatrix}1 & \alpha \\ 0 & 1\end{pmatrix} \in U^n$ and $x \in Y_n$ such that $u^n x \in Y_n$. We will first consider the case that $x \in \Sigma$, so that $x = (\ell_\infty(a), \ell_0(b))$ for some $a,b \in \RR$. If $u^n \neq \operatorname{Id}$, then either $v_\infty(\alpha)\leq -(n+1)$ or $v_0(\alpha) \leq -(n+1)$. Therefore either
\[u^n \cdot L_\infty \cap \pi_\infty(Y_n) \subseteq \ell_\infty([n+1,\infty))\]
or
\[u^n \cdot L_0 \cap \pi_0(Y_n) \subseteq \ell_0([n+1,\infty))\]
where $\pi_\infty$ and $\pi_0$ are orthogonal projections to $T_\infty$ and $T_0$ respectively. Therefore if $u^n x \in Y_n$, we have $a,b \geq n+1$ and $u^n x = x$.

Now consider general $x \in Y_n$. By the definition of $Y_n$, $x=vy$ for some $v \in \operatorname{Stab}_U(x_n)$ and $y \in \Sigma$. Since $U$ is abelian, $u^n x = v u^n y$. Since $v u^n y \in Y_n$, so is $u^n y$. As above, this implies that $y = u^n y$ and thus that $x = u^n x$.

\end{proof}

We may identify $Y^n$ with the quotient $U^n \backslash X$. Let $\theta_n:X \to Y_n$ be the quotient map. Note that $\theta_n$ is $U_n$-equivariant since $U$ is abelian.

We may now define $\Phi_n \in H^2(\Gamma \backslash X_3; F)$ by 
\begin{equation*}
\Phi_n(\Gamma D) = \sum_{U_\Gamma g \in U_\Gamma \backslash \Gamma} \varphi_n (\theta_n(\psi(gD)) \cap S_n^\downarrow)
\end{equation*}
for any 2-cell $D$ in $X_3$.
The next several lemmas show that $\Phi_n$ is well-defined as a formal sum, always contains a finite number of nonzero terms, is a cocycle, and is $U_n$-invariant.

\begin{lemma} \label{welldef}
$\Phi_n$ is well-defined. In particular, it is independent of the choice of coset representatives in $\Gamma D$ and $U_\Gamma g$.
\end{lemma}
\begin{proof}
This proof closely follows the proof of  a lemma from \cite{CesaKelly2013}.
Let $u \in U_\Gamma, u=u^n u_n$ for $u^n \in U^n$ and $u_n \in U_n$. Then
\begin{align*}
\varphi_n(\theta_n (\psi(ugD))\cap S_n^\downarrow) 
&= \varphi_n(\theta_n (u \psi(gD))\cap S_n^\downarrow) \\
&= \varphi_n( \theta_n (u^n u_n \psi(gD)) \cap S_n^\downarrow) \\
&= \varphi_n( \theta_n (u_n \psi(gD)) \cap S_n^\downarrow) \\
&= \varphi_n( u_n \theta_n (\psi(gD)) \cap S_n^\downarrow) \\
&= \varphi_n( \theta_n (\psi(gD)) \cap S_n^\downarrow)
\end{align*}
So $\Phi_n$ is independent of coset representative in $U_\Gamma g$.

Let $\gamma \in \Gamma$. Then
\begin{align*}
\Phi_n(\Gamma (\gamma D)) &= \sum_{U_\Gamma g \in U_\Gamma \backslash \Gamma} \varphi_n (\theta_n (\psi(g\gamma D)\cap S_n^\downarrow)) \\
&= \sum_{U_\Gamma (g \gamma) \in U_\Gamma \backslash \Gamma} \varphi_n (\theta_n (\psi(g \gamma D)\cap S_n^\downarrow)) \\
&= \sum_{U_\Gamma g \in U_\Gamma \backslash \Gamma} \varphi_n (\theta_n (\psi(g D)\cap S_n^\downarrow)) \\
&= \Phi_n(\Gamma D)
\end{align*}
Therefore $\Phi_n$ is also independent of coset representative in $\Gamma D$.
\end{proof}

\begin{lemma} \label{finite-sum}
For any 2-cell $D$ in $X_3$, the formal sum 
\begin{equation*}
\Phi_n(\Gamma D) = \sum_{U_\Gamma g \in U_\Gamma \backslash \Gamma} \varphi_n (\theta_n(\psi(gD)) \cap S_n^\downarrow)
\end{equation*}
has a finite number of nonzero terms.
\end{lemma}

\begin{proof}
Since $D$ is compact, so is $\psi(D)$. Therefore $\psi(D)$ must intersect a finite number of $\Gamma$-translates of $H$. We may therefore write
\begin{equation*}
D = \left(\bigcup_{i = 1}^k (D \cap \psi^{-1}(\alpha_i H))\right) \cup \overline{D}
\end{equation*}
where $k$ is some finite number, $\alpha_i \in \Gamma$ and $\overline{D}$ is disjoint from the interior of $\Gamma H$.

Let $D_i = D \cap \psi^{-1}(\alpha_i H)$ 

Then we have
\begin{equation*}
\Phi_n(\Gamma D) = \Phi_n(\Gamma \overline{D}) + \sum_{i=1}^k\Phi_n(\Gamma D_i) 
\end{equation*}

Since $\theta_n(\psi(g\overline{D}) \cap S_n^\downarrow = \emptyset$ for every $g \in \Gamma$, we have $\Phi_n(\Gamma \overline{D})=0$. Therefore, it suffices to show that $\Phi_n(\Gamma D_i)$ is a finite sum for each $i$. Let 
\begin{equation*}
C_i = \left\{\left. \gamma \in \Gamma \right\vert \gamma \psi(D_i) \cap S_n^\downarrow \neq \emptyset  \right\}
\end{equation*}
and
\begin{equation*}
C_i^\prime = \left\{\left.U_\Gamma \gamma \in U_\Gamma \backslash \Gamma \right\vert \xi \in C_i \text{ for some } \xi \in U_\Gamma \gamma \right\}
\end{equation*}
We will show that $C_i^\prime$ is finite for each $i$.

Suppose first that $\alpha_i H = H$. Then $C_i \subset P_\Gamma = U_\Gamma A_\Gamma$. Suppose $ua \in C$, with $u \in U_\Gamma$ and $a \in A_\Gamma$. Since either $uS_n^\downarrow \cap S_n^\downarrow = \emptyset$ or $u$ fixes $S_n^\downarrow$ pointwise, we may conclude that $a \in C$. Since $\psi(D_i)$ is compact and $a S_n^\downarrow \cap S_n^\downarrow = \emptyset$ for every $a \in A$, $C_i$ contains finitely many elements of $A$. Thus $C^\prime_i$ is finite.

Now consider the general $\alpha_i$ with nonempty $C_i$ and let $\gamma \in C_i$. Since $\Phi_n(\Gamma \gamma D_i) = \Phi_n(\Gamma D_i)$ and $\gamma D_i \subset H$, the result follows from the argument above.

\end{proof}

\begin{lemma} \label{cocycle}
$\Phi_n$ is a cocycle.
\end{lemma}
\begin{proof}
Let $\Gamma D$ be a 3-cell in $\Gamma \backslash X_3$ corresponding to the 3-cell $D$ in $X_3$. Then $\partial D$ is a 2-sphere. Since $X$ contains no nontrivial 2-spheres, $\psi(\partial D)$ is trivial. Thus, 
\begin{equation*}
\varphi_n (\theta_n(\psi(g\partial D)) \cap S_n^\downarrow)= 0
\end{equation*}
 for any $g \in \Gamma$ and $\Phi_n(\Gamma D) = 0$.

\end{proof}
\begin{lemma} \label{Un-invariant}
$\Phi_n$ is $U_n$-invariant.
\end{lemma}
\begin{proof}
It suffices to show that $\theta_n \psi(D) \cap S_n$ is supported on $S_n^\downarrow$ for every disk $D$ in $X_3$ since $\varphi_n$ is $U_n$-invariant.

Let $H_n = \beta_\rho^{-1}([\beta_\rho(x_n),\infty))$. Since $H_n \subset H$, the definition of $\psi$ implies that $\psi(X_1) \cap H_n = \emptyset$. Thus, $\partial \psi(D)$ is entirely outside of $H_n$. Since $U^n$ preserves $H_n$, we may also conclude that $\partial \theta_n(\psi(D))$ is outside $\theta_n(H_n)$.

Suppose there is some cell $C_0 \subset \op{Supp}(\theta_n(\psi(D))$ with $C_0 \subset S_n$ and $C_0 \not\subset S_n^\downarrow$. We will find an infinite family of cells which must also be contained in $\op{Supp}(\theta_n(\psi(D))$, contradicting the compactness of $\theta_n(\psi(D))$.

Since $S_n \subset U_n \Sigma$ and $S_n^\downarrow$ is $U_n$-invariant, we may assume without loss of generality that $C_0 \subset \Sigma$. Then $C_0$ is one of $\ell_\infty([n-1,n]) \times \ell_0([n,n+1])$, $\ell_\infty([n,n+1]) \times \ell_0([n-1,n])$, or $\ell_\infty([n,n+1]) \times \ell_0([n,n+1])$. We will examine in detail the case where $C_0 = \ell_\infty([n-1,n]) \times \ell_0([n,n+1])$. The others are parallel.

Let $C_i = \ell_\infty([n-1,n]) \times \ell_0([n+i,n+i+1])$ and $e_i = \ell_\infty([n-1,n]) \times \ell_0(n+i)$ so that $C_i, e_i \subset \Sigma$ with $C_{i-1}$ and $C_i$ adjacent along $e_i$. We will show by induction that all $C_i$ are in $\op{Supp}(\theta_n(\psi(D))$.

Since $e_i \subset H_n$, we have $e_i \not\subset \partial \theta_n(\psi(D))$. Thus, if $\op{Supp}(\theta_n(\psi(D))$ contains $C_{i-1}$, it must also contain some other cell with $e_i$ in its boundary. Such a cell must be of the form $\ell_\infty([n-1,n]) \times e$ for some edge $e \subset T_0$ with $\ell_0(n+i)$ one endpoint. Since the action of $U_n$ on $T_0$ fixes all edges incident to $\ell_0(n+i)$, the only cells of this form in $U_n \Sigma$ are the ones in $\Sigma$ itself: $C_{i-1}$ and $C_i$. Thus $C_i \subset \op{Supp}(\theta_n(\psi(D))$.

This shows that $\op{Supp}(\theta_n(\psi(D))$ contains infinitely many cells and is not compact.
\end{proof}

Now that we have defined a family of cocyles, it remains to show that they are independent. We will do this by exhibiting a family of disks $\widetilde{B_{2n}}\subset X_3$ such that $\Phi_{2n}(\Gamma \widetilde{B_{2n}})=1$ and  $\Phi_k(\Gamma \widetilde{B_{2n}})=0$ for $k > 2n$.

Let $Z_{2n}$ be the triangle in $\Sigma$ with vertices $x_n$, $(\ell_\infty(n),\ell_0(-n))$, and $(\ell_\infty(-n),\ell_0(n))$. Let
\begin{equation*}
B_{2n} = Z_{2n} - \begin{pmatrix} 1 & t^{2n} \\ 0 & 1 \end{pmatrix} Z_{2n} - \begin{pmatrix} 1 & t^{-2n} \\ 0 & 1 \end{pmatrix} Z_{2n} + \begin{pmatrix} 1 & t^{-2n}+t^{2n} \\ 0 & 1 \end{pmatrix} Z_{2n}
\end{equation*}

Then $B_{2n}$ is a square with $B_{2n} \cap S_{2n}^\downarrow = C^{2n}_{0,0} - C^{2n}_{1,0} - C^{2n}_{0,1} + C^{2n}_{1,1}$.

\begin{lemma}\label{lift}
For every $n \in \NN$, there is some disc $\widetilde{B_{2n}} \subset X_3$ such that $\psi(\widetilde{B_{2n}}) \cap H = B_{2n}$. 
\end{lemma}
\begin{proof}
Letting $u = \begin{pmatrix} 1 & t^{-2n}+t^{2n} \\ 0 & 1 \end{pmatrix}$ and $a = \begin{pmatrix} t^n & 0 \\ 0 & t^{-n} \end{pmatrix}$, the vertices of $B_{2n}$ are $a x_0$, $a^{-1}x_0$, $ua x_0$, and $ua^{-1}x_0$. In particular, these four vertices are in $\Gamma x_0$ and therefore in $\psi(X_0)$. Let $v_i$ and $v_j$ be two vertices on an edge of the square, with $\widetilde{v_i} \in \psi^{-1}(v_i)$ and $\widetilde{v_j} \in \psi^{-1}(v_j)$. Since $X_3$ is connected, there is a path $\widetilde{c_{i,j}}$ between $\widetilde{v_i}$ and $\widetilde{v_j}$. Thus $c_{i,j} = \psi(\widetilde{c_{i,j}})$ is a path connecting $v_i$ and $v_j$. By the definition of $\psi$, $c_{i,j}$ is disjoint from $\Gamma H$. Repeating this for all four edges gives a path $c$ in $X$ which bounds a disk $B_{2n}^\prime$ such that $B_{2n} \cap H = B_{2n}^\prime \cap H$ and is the image of a path $\widetilde{c}$ in $X_3$. Since $X_3$ is 2-connected, there is a disk $\widetilde{B_{2n}}$ that fills $\widetilde{c}$, and $\psi(\widetilde{B_{2n}}) = B_{2n}^\prime$.
\end{proof}
 
\begin{lemma}\label{independence}
$\Phi_{2n}(\Gamma \widetilde{B_{2n}})=1$ and $\Phi_k(\Gamma \widetilde{B_{2n}})=0$ for $k>2n$.
\end{lemma}
\begin{proof}

We will prove that $\Phi_k(\Gamma \widetilde{B_{2n}})=0$ by showing that $\gamma B_{2n} \cap S_k^\downarrow = \emptyset$ for every $\gamma \in \Gamma$. Since $B_{2n} \cap H \subset \beta_\rho^{-1}([R+d,2n])$ and $P_\Gamma$ preserves $\beta_\rho$, we have $pB_{2n}\cap S_k^\downarrow = \emptyset$ for every $p \in P_\Gamma$. For $\gamma \in \Gamma - P_\Gamma$, $\gamma H \cap H = \emptyset$. Since $B_{2n} \cap \Gamma H \subset H$, this implies that $\gamma B_{2n} \cap S_{k}^\downarrow = \emptyset$.

By the definitions of $B_{2n}$ and $\varphi_{2n}$, we have $\varphi_{2n}(B_{2n})=1$. It therefore suffices to show that $\gamma B_{2n} \cap S_{2n} = \emptyset$ for every $\gamma \in \Gamma - U_\Gamma$.

Let $p \in P_\Gamma - U_\Gamma$. Since $P_\Gamma$ preserves $\beta_\rho$ and $B_{2n} \cap \beta_\rho^{-1}([2n-1,2n+1]) \subset S_{2n}$, it follows that $pB_{2n} \cap \beta_\rho^{-1}([2n-1,2n+1]) \subset p S_{2n}$. Let $p=au$ for some $a \in A_\Gamma$ and $u \in U_\Gamma$ with $a \neq \id$. Since $uS_{2n} = S_{2n}$ and $aS_{2n} \cap S_{2n} = \emptyset$, we have $pB_{2n} \cap S_{2n} = \emptyset$.

For $\gamma \in \Gamma - P_\Gamma$, $\gamma H \cap H = \emptyset$. Since $B_{2n} \cap \Gamma H \subset H$, this implies that $\gamma B_{2n} \cap S_{2n} = \emptyset$.

\end{proof}

Thus, for any $k$ and any even integer $m \geq R+d$, the set of cocycles \[\left\{ \Phi_{m+2}, \Phi_{m+4}, \dots \Phi_{m+2k}\right\}\] is independent, which suffices to prove Theorem \ref{mainthm}.

\bibliographystyle{amsalpha}
\bibliography{../general.bib}

\Addresses

\end{document}